\documentclass[a4paper,12pt]{amsart}
\usepackage[cp1251]{inputenc}
\usepackage[ russian, ukrainian, english]{babel}
\usepackage{amsmath, amsthm, amsfonts, amssymb}

\newcommand{\wt}{\widetilde}
\newcommand{\Ker}{\mathop{\rm ker}}
\renewcommand{\dim}{\mathop{\rm dim}}
\newcommand{\ov}{\overline}

\newcommand{\ve}{\varepsilon}
\newcommand{\1}{1\!\!\,{\rm I}}
\newcommand{\mbR}{{\mathbb R}}

\newcommand{\cK}{{\mathcal K}}
\newcommand{\cT}{{\mathcal T}}

\theoremstyle{plain}
\newtheorem{thm}{Theorem}
\newtheorem{lem}{Lemma}
\theoremstyle{definition}
\newtheorem{defn}{Definition}

\begin{document}
 \makeatletter
 \headsep 10 mm
 \footskip 20 mm
\renewcommand{\@evenhead}%
{\vbox{\hbox to\textwidth{\strut {\thepage \hfil  { On self-intersection local times ...}}}
\hrule}}
\renewcommand{\@oddhead}%
{\vbox{\hbox to\textwidth{\strut {{  A.A.Dorogovtsev, O.L.Izyumtseva} \hfil\thepage}} \hrule}} \makeatother

\begin{center}
{\Large\bf On self-intersection local times for generalized Brownian bridges and the distance between step functions}\\[1cm]
A.A.Dorogovtsev and O.L.Izyumtseva
\end{center}
\vskip25pt

\noindent
{\bf Subjclass}: 60G15, 60H40, 60J55, 46C05.
\vskip25pt
\centerline{\bf{Abstract}}
\vskip15pt
In the paper $k$-multiple self-intersection local time for planar Gaussian integrators generated by linear operator with nontrivial kernel is studied. In this case additional singularities arise in its formal Fourier--Wiener transform. In case $k=2$ the set of singularities is the finite number of points. In case $k>2$ it contains intervals and hyperplanes. In the first and the second cases using two different approaches related on structure of set of singularities we show that ``new'' singularities do not imply on the convergence of integral corresponding to the formal Fourier--Wiener transform and regularization consist of compensation of impact of diagonals as for the Wiener process.

\vskip20pt
\begin{center}
{\bf 1. Introduction}
\end{center}
\vskip15pt
In present article we study a self-intersection local time (SILT) for planar Gaussian integrators
\begin{equation}
\label{eq1}
x(t)=((A\1_{[0;t]}, \xi_1), (A\1_{[0;t]}, \xi_2)),  \ t\in[0; 1].
\end{equation}
Here $A$ is a continuous linear operator in $L_2([0; 1]), \xi_1, \xi_2$ are two independent Gaussian white noises in the same space \cite{1, 2}. Gaussian integrators firstly appeared in works of A.A.Dorogovtsev \cite{3, 4} in connection with an anticipating stochastic integration. Note that if $A$ is an identity formula \eqref{eq1} defines a planar Wiener process. For $A=I-P,$ where $P$ is a projection onto $\1_{[0; 1]}$ the process $x$ is a planar Brownian bridge. One can check that planar fractional Brownian motion with Hurst parameter $\alpha>\frac{1}{2}$ has representation \eqref{eq1} with integral operator $A$ defined by kernel $K(t_1, t_2)=(t_2-t_1)^{2\alpha-2}\1_{\{t_2>t_1\}}$ (see \cite{5} for the proof). $k$-multiple SILT for the process $x$ is formally defined as
\begin{equation}
\label{eq2}
T^x_k=\int_{\Delta_k}\prod^{k-1}_{i=1}\delta_0(x(t_{i+1})-x(t_i))d\vec{t},
\end{equation}
where $\Delta_k=\{0\leq t_1\leq\ldots\leq t_k\leq1\},$ $\delta_0$ is a delta-function at the point zero. It is well known
[6--8] that \eqref{eq2} can not be defined as the limit of approximating family
$$
T^x_{\ve,k}=\int_{\Delta_k}\prod^{k-1}_{i=1}f_\ve(x(t_{i+1})-x(t_i))d\vec{t}
 $$
 with $f_\ve(z)=\frac{1}{2\pi\ve}e^{-\frac{\|z\|^2}{2\ve}},\ z\in\mbR^2,\  \ve>0$ even in the case of planar Wiener process $w.$ Different approaches to a renormalization of $T^w_k$ where described in [6--8]. The application of renormalized self-intersection local time for planar Wiener process is considered in \cite{9}. Most related to our work is Rosen renormalization \cite{7}. Consider it more precisely. J.Rosen introduced the following renormalization
 $$
R^w_{\ve,k}=\int_{\Delta_k}\prod^{k-1}_{i=1}f_\ve(w(t_{i+1})-w(t_i))d\vec{t},
$$
where $\{\eta\}=\eta-E\eta$ and proved the convergence in mean square of random variables $R_{\ve, k}$ as $\ve\to0.$

The renormalization of SILT for planar Gaussian processes which do not have Markov property is systematically studied in works of the authors [10--13]. The white noise tools allow to ignore Wiener properties and reduce the consideration of finite dimensional distributions of small increments to the studying of geometry of Gram determinants constructed by increments of Hilbert valued function generating the process and projections on its linear spans. For the description of functionals from $\xi_1, \xi_2$ we use the Fourier--Wiener transform. It is known that any square integrable random variable $\alpha$ which is measurable with respect to white noises $\xi_1,\xi_2$ uniquely defined by its Fourier--Wiener transform \cite{14}. In the paper we use the following definition.
\begin{defn}
\label{defn1}
$$
\cT(\alpha)=E\alpha e^{(h_1,\xi_1)+(h_2, \xi_2)-\frac{1}{2}(\|h_1\|^2+\|h_2\|^2)},\ h_1,h_2\in L_2([0; 1])
$$
is said to be the Fourier--Wiener transform of random variable $\alpha$.
\end{defn}
By calculating
$$
ET^x_{\ve, k}e^{(h_1,\xi_1)+(h_2, \xi_2)-\frac{1}{2}(\|h_1\|^2+\|h_2\|^2)}
$$
and formally passing to the limit as $\ve\to0$ one can get the formal Fourier--Wiener transform of random variable $T^x_k$ which is described by expression
\begin{equation}
\label{eq3}
\int_{\Delta_k}\frac{1}
{(2\pi)^{k-1}G(\Delta g(t_1), \ldots, \Delta g(t_{k-1}))}
e^{-\frac{1}{2}(\|P_{t_1\ldots t_k}h_1\|^2+P_{t_1\ldots t_k}h_2\|^2)}d\vec{t},
\end{equation}
where $G(\Delta g(t_1), \ldots, \Delta g(t_{k-1}))$ is the Gram determinant constructed from increments of function $g(t)=A\1_{[0; t]}$ and $P_{t_1\ldots t_k}$ is a projection onto linear span generated by \break $\Delta g(t_1), \ldots, \Delta g(t_{k-1}).$ Further in the paper for the linear span generated by elements $q_1,\ldots,q_m$ of $L_2([0;1])$ we use notation $LS\{q_1, \ldots,q_m\}$. Note that \eqref{eq3} is divergent integral since the denominator blow up on the diagonals of $\Delta_k.$ The renormalization for $T^x_k$ is equivalent to the regularization of divergent integral \eqref{eq3}. Such regularization for \eqref{eq3} was introduced by authors in \cite{10} for $A=I+S,$ where $I, S$ are identity and compact operators in $L_2([0; 1]),\ \|S\|<1.$ Condition $\|S\|<1$ implies a continuous invertibility of operator $I+S.$ In this case the Gram determinant in \eqref{eq3} turns to zero only on diagonals of $\Delta_k.$ In \cite{10} the following regularization was proposed. Let $\Delta\wt{g}(t_1), \ldots, \Delta\wt{g}(t_{k-1})$ be an orthonormal system which is obtained from $\Delta{g}(t_1), \ldots, \Delta{g}(t_{k-1})$ via the orthogonalization procedure. For $M\subset\{1, \ldots, k-1\}$ denote by $P_M$ the projection onto $LS\{\Delta\wt{g}(t_i),\ i\in M\}.$
\begin{thm}[\cite{10}]
\label{thm1}
For any $h\in L_2([0; 1])$ the following integral
\begin{equation}
\label{eq4}
\int_{\Delta_k}\frac{1}
{G(\Delta g(t_1), \ldots, \Delta g(t_{k-1}))}
\sum_{M\subset\{1, \ldots, k-1\}}(-1)^{|M|}
e^{-\frac{1}{2}\|P_Mh\|^2}d\vec{t}
\end{equation}
converges.
\end{thm}
One can see that regularization \eqref{eq4} for \eqref{eq3} coincides with Rosen renormalization for $T^x_k$ in case $A=I.$ The aim of present paper is to consider the general case when the operator which generates the integrator $x$ has a nontrivial kernel. In this case additional singularities arise in integral \eqref{eq3}.
We consider the case when $\dim \Ker A<+\infty.$ Such condition leads to the concretization of additional singularities in \eqref{eq3}. Namely, for $k=2$ one can check that the Gram determinant in the denominator can have the new zeros only at the finite number of points. We specify the asymptotics of the denominator in these points and check the convergence of the integral. For $k>2$ the set of singularities of $\cT(T^x_k)$ has a complicated structure. It contains intervals and hyperplanes. That is why described approach can not be extended on cases $k>2.$ For $k>2$ the method we use rely on studying of functional properties of function $G(\Delta g(t_1), \ldots, \Delta g(t_{k-1})),\ t_1, \ldots, t_k\in\Delta_k,$ where as it was mentioned above $G$ is the Gram determinant constructed from increments of function $g(t)=A\1_{[0; t]}.$ Condition $G(\Delta g(t_1), \ldots, \Delta g(t_{k-1}))=0$ as it will be discussed further implies that step functions belong to $\Ker A.$ Proving the positivity of distances between orthogonal complement of subspace of step functions in $\Ker A$  and subspaces generated by step functions and indicators $\1_{[t_i; t_{i+1}]},\  i=\overline{1, k-1}$ we obtain lower estimates for $G(\Delta g(t_1), \ldots, \Delta g(t_{k-1}))$ which allow to check the convergence of integral corresponding to the Fourier--Wiener transform of $k$-multiple SILT on the domain off the diagonals for planar Gaussian integrator.
\vskip15pt
\begin{center}
{ \bf 2. Double self-intersection local time of generalized Brownian bridges}
\end{center}
\label{section2}
\vskip15pt
In this section we study double SILT for process \eqref{eq1} with $A=I+S,$ where $I,\ S$ are identity and compact operators in the space $L_2([0;1])$. The formal Fourier--Wiener transform \eqref{eq3} for that case has the following representation
\begin{equation}
\label{eq5}
\cT(T^x_2)(h_1, h_2)=
\int_{\Delta_2}\frac{1}
{2\pi\|(I+S)\1_{[t_1;t_2]}\|^2}
e^{-\frac{1}{2}(\|P_{t_1t_2}h_1\|^2+\|P_{t_1t_2}h_2\|^2)}d\vec{t}.
\end{equation}
As it was mentioned in Introduction in case $\Ker I+S\ne \{0\}$ the denominator in integral \eqref{eq5} can have zeros outside of diagonal. We check that ``old'' regularization \cite{10} for $\cT(T^x_2)$  remains valid and new singularities do not influence on the integrability of function $\frac{1}{\|(I+S)\1_{[t_1;t_2]}\|^2},\ t_1,t_2\in\Delta_2$ on any subset of $\Delta_2$ off the diagonal. Let L be a finite-dimensional subspace of $L_2([0; 1])$. Denote by $E_L=\{\1_{[t_1; t_2]}\in L,\ t_1, t_2\in\Delta_2, \ t_1<t_2\}.$
\begin{lem}
\label{lem1}
Set $E_L$ is finite.
\end{lem}
\begin{proof}
Suppose that $\dim L=m.$ Then set $E_L$ contains at most $m$ linearly independent elements. Denote by $\1_{[t^1_1; t^1_2]}, \ldots, \1_{[t^l_1; t^l_2]}$ the maximal linearly independent subset of it. Let $\1_{[u; v]}\in LS\{\1_{[t^1_1; t^1_2]}, \ldots, \1_{[t^l_1; t^l_2]}\}.$ Then one can check that $u, v\in\{t^1_1, t^1_2, \ldots, t^l_1, t^l_2\}.$ This completes the proof of the lemma.
\end{proof}
The following two statements related to the behaviour of $\|\1_{[t_1; t_2]}\|$ in small neighborhoods of elements from $E_L.$
\begin{lem}
\label{lem2}
Let $\1_{[t^0_1; t^0_2]}$ be a fixed element of $E_L$. Then the family of functions
$$
\left\{
\frac{\1_{[t_1;t_2]}-\1_{[t^0_1; t^0_2]}}
{\|
\1_{[t_1;t_2]}-\1_{[t^0_1; t^0_2]}
\|},\ t_1, t_2\in\Delta_2
\right\}
$$
weakly converges to zero in $L_2([0; 1])$ as $t_1\to t^0_1,\ t_2\to t^0_2.$
\end{lem}
\begin{proof}
Let us check that for any $h\in L_2([0; 1])$
$$
\left(h,
\frac{\1_{[t_1;t_2]}-\1_{[t^0_1; t^0_2]}}
{\|
\1_{[t_1;t_2]}-\1_{[t^0_1; t^0_2]}
\|},\ t_1, t_2\in\Delta_2
\right)\to0
$$
as $t_1\to t^0_1,\ t_2\to t^0_2.$ Denote by
$$
e(\vec{t})=\1_{[t_1; t_2]}-\1_{[t^0_1; t^0_2]},
$$
$$
e_i(\vec{t})=\1_{[t_i\wedge t^0_i; t_i\vee t^0_i]},\ i=1,2.
$$
Since
\begin{equation}
\label{eq6}
\|P_{e(\vec{t})}h\|^2=\frac{(h, e(\vec{t}))^2}
{\|e(\vec{t})\|^2},
\end{equation}
then to prove the statement it suffices to check that for any $h\in L_2([0; 1])$
$$
\|P_{e(\vec{t})}h\|^2\to0, \ t_1\to t^0_1,\ t_2\to t^0_2.
$$
Note that $e(\vec{t})\in LS\{e_1(\vec{t}), e_2(\vec{t})\}.$ Consequently,
$$
\|P_{e(\vec{t})}h\|^2\leq\|P_{e_1(\vec{t})e_2(\vec{t})}h\|^2.
$$
Let $O_\delta(t)=(t-\delta; t+\delta).$ Then there exists $\delta>0$ such that for any $t_1\in O_\delta(t^0_1),\  t_2\in O_\delta(t^0_2)$ $e_1(\vec{t})$ and $e_2(\vec{t})$ are orthogonal. It implies that
$$
P_{e_1(\vec{t})e_2(\vec{t})}=\sum^2_{i=1}P_{e_i(\vec{t})}.
$$
Note that for any $h\in L_2([0; 1]),\  i=1,2$
\begin{equation}
\label{eq7}
\|P_{e_i(\vec{t})}h\|^2\to0,  \ t_i\to t^0_i.
\end{equation}
Really,
\begin{equation}
\label{eq8}
\|P_{e_i(\vec{t})}h\|^2=
\frac{(e_i(\vec{t}),h)^2}{\|e_i(\vec{t})\|^2}=
\frac
{\Big(\int^{t_i\vee t^1_i}_{t_i\wedge t^1_i}h(s)ds\Big)^2}
{(t_i\vee t^1_i-t_i\wedge t^1_i)}.
\end{equation}
The Cauchy inequality imply that \eqref{eq8} is less or equal to $\int^{t_i\vee t^1_i}_{t_i\wedge t^1_i}h^2(s)ds.$
By continuity of Lebesgue integral the last expression tends to zero when $t_i\to t^0_i.$
\end{proof}
\begin{lem}
\label{lem3}
Let $S$ be a compact operator in $L_2([0; 1])$ and $\1_{[t^0_1;t^0_2]}\in\Ker I+S.$ Then
$$
\frac{\|(I+S)\1_{[t_1;t_2]}\|^2}
{\|\1_{[t_1;t_2]}-\1_{[t^0_1; t^0_2]}\|^2}\to1,\ t_1\to t^0_1,\ t_2\to t^0_2.
$$
\end{lem}
\begin{proof}
Note that
$$
\|(I+S)\1_{[t_1;t_2]}\|^2=\|(I+S)(\1_{[t_1;t_2]}-\1_{[t^0_1; t^0_2]})\|^2.
$$
It implies that
$$
\frac{\|(I+S)\1_{[t_1;t_2]}\|^2}
{\|\1_{[t_1;t_2]}-\1_{[t^0_1; t^0_2]}\|^2}=
$$
$$
=1-2
\frac{(\1_{[t_1;t_2]}-\1_{[t^0_1; t^0_2]}, S(\1_{[t_1;t_2]}-\1_{[t^1_1; t^1_2]}))}
{\|\1_{[t_1;t_2]}-\1_{[t^0_1; t^0_2]}\|^2}+
$$
$$
+
\frac{\| S(\1_{[t_1;t_2]}-\1_{[t^0_1; t^0_2]})\|^2}
{\|\1_{[t_1;t_2]}-\1_{[t^0_1; t^0_2]}\|^2}.
$$
The compactness of operator $S$ and Lemma \ref{lem2} end the proof.
\end{proof}
The following statement expands ``old'' regularization \cite{10} for $\cT(T^x_2)$ on planar Gaussian integrator generated by noninvertible operator $I+S,$ where $I, S$ are identity and compact operators in $L_2([0; 1]).$
\begin{thm}
\label{thm2}
For any $h\in L_2([0; 1])$ the following integral
$$
\int_{\Delta_2}
\frac{1}
{\|(I+S)\1_{[t_1;t_2]}\|}
[e^{-\frac{1}{2}\|P_{t_1t_2}h\|^2}-1]d\vec{t}
$$
converges.
\end{thm}
\begin{proof}
It follows from Lemma \ref{lem1} that set $E_{\Ker I+S}$ is finite. Suppose that
$E_{\Ker I+S}=\{\1_{[t^1_1; t^1_2]}, \ldots, \1_{[t^m_1;t^m_2]}\}.$ One can conclude from Lemma \ref{lem3} that for any $i=\ov{1, m},\ \ve>0$ there exists $\delta_i$ such that for any $t_1\in O_{\delta_i}(t^i_1),\  t_2\in O_{\delta_i}(t^i_2)$
\begin{equation}
\label{eq9}
\|(I+S)\1_{[t_1;t_2]}\|^2\geq(1-\ve)\|\1_{[t_1;t_2]}-\1_{[t^i_1; t^i_2]}\|^2.
\end{equation}
Put $\delta=\delta_1\wedge\ldots\wedge\delta_m.$ Consider nonintersecting domains
$K_i=\ov{O}_\delta(t^i_1)\times\ov{O}_\delta(t^i_2)\bigcap\Delta_2,\ \ov{1, m}.$ It follows from \cite{10} that for any $h\in L_2([0;1])$
$$
\int_{\Delta_2\setminus\cup^m_{i=1}K_i}
\frac{1}
{\|(I+S)\1_{[t_1;t_2]}\|^2}
[e^{-\frac{1}{2}\|P_{t_1t_2}h\|^2}-1]d\vec{t}
$$
converges. Therefore to end the proof it suffices to check that
$$
\int^{t^1_2+\delta}_{t^1_2-\delta}\int^{t^1_1+\delta}_{t^1_1-\delta}
\frac{1}
{\|(I+S)\1_{[t_1;t_2]}\|^2}dt_1dt_2
$$
converges. Relation \eqref{eq9} implies that
$$
\int^{t^1_2+\delta}_{t^1_2-\delta}\int^{t^1_1+\delta}_{t^1_1-\delta}
\frac{1}
{\|(I+S)\1_{[t_1;t_2]}\|^2}dt_1dt_2\leq
$$
$$
\leq
\frac{1}{1-\ve}
\int^{t^1_2+\delta}_{t^1_2-\delta}\int^{t^1_1+\delta}_{t^1_1-\delta}
\frac{1}
{\|\1_{[t_1;t_2]}-\1_{[t^1_1; t^1_2]}\|^2}dt_1dt_2=
$$
$$
=
\int^{t^1_2}_{t^1_2-\delta}\int^{t^1_1}_{t^1_1-\delta}
\frac{dt_1dt_2}{t^1_1-t_2-t_1+t^1_2}+
$$
$$
+\int^{t^1_2+\delta}_{t^1_2}\int^{t^1_1}_{t^1_1-\delta}
\frac{dt_1dt_2}{t_2-t_1-t^1_2+t^1_1}+
$$
$$
+
\int^{t_2}_{t^1_2-\delta}\int^{t^1_1-\delta}_{t^1_1}
\frac{dt_1dt_2}{t^1_2-t^1_1-t_2+t_1}+
$$
$$
+
\int^{t^1_2+\delta}_{t^1_2}\int^{t^1_1+\delta}_{t^1_1}
\frac{dt_1dt_2}{t_2+t_1-t^1_2-t^1_1},
$$
where each summand is finite.
\end{proof}
\vskip15pt
\begin{center}
{\bf 3. $k$-multiple self-intersection local time of generalized Brownian bridges}
\end{center}
\label{section3}
\vskip15pt
The main object of investigation in this section is $k$-multiple SILT for planar Gaussian integrators generated by continuous linear operator $A$ in $L_2([0; 1])$ which satisfies conditions

1) $\dim\Ker A<+\infty$

2) The restriction of operator $A$ on orthogonal complement of $\Ker A$ is continuously invertible operator.

 Let us notice that such class of Gaussian processes contains planar Gaussian integrators generated by $I+S,$ where $S$ is a compact operator in $L_2([0; 1]).$ As it was mentioned in Introduction the set of zeros of function
 $$
 G(\Delta g(t_1), \ldots, \Delta g(t_{k-1})),\ t_1, \ldots, t_k\in\Delta_k,\ k>2
 $$
has a complicated structure. Here as above $G(\Delta g(t_1), \ldots, \Delta g(t_{k-1}))$
is a Gram determinant constructed by $\Delta g(t_1), \ldots, \Delta g(t_{k-1}), \ g(t)=A\1_{[0; t]}.$ It contains intervals and hy\-per\-pla\-nes. The approach we use in case $k>2$ rely on studying geometric properties of function $G(\Delta g(t_1), \ldots, \Delta g(t_{k-1})),\ t_1, \ldots, t_k\in\Delta_k.$ We will see that condition
$$
G(\Delta g(t_1), \ldots, \Delta g(t_{k-1}))=0
$$
generates certain sets of step functions. Analyzing the distances between those sets we obtain lower estimates for $G(\Delta g(t_1), \ldots, \Delta g(t_{k-1}))$ which allow to establish that function
$$
 \frac{1}
{G(\Delta g(t_1), \ldots, \Delta g(t_{k-1}))},\ t_1, \ldots, t_k\in\Delta_k
$$
is integrable on any $\Delta^\delta_k,$ where $\Delta^\delta_k=\{\vec{t}\in\Delta_k, t_{i+1}-t_i\geq\delta\},\  \delta>0.$
It inspires to make a conclusion that ``old'' renormalization \cite{10} remains true for $\cT(T^x_k)$ in case $k>2$ too. We did not check that in this section but it will be the object of our further considerations. The main statement of present section is the following theorem.
\begin{thm}
\label{thm3}
Suppose that continuous linear operator $A$ in $L_2([0; 1])$ satisfies conditions 1), 2). Then for any $\delta>0$
$$
\int_{\Delta^\delta_k}
\frac{1}
{G(\Delta g(t_1), \ldots, \Delta g(t_{k-1}))}d\vec{t}<+\infty.
$$
\end{thm}

To prove the statement we need the following lemma.
\begin{lem}
\label{lem4}
Let $\dim L<+\infty,\  e_1, \ldots, e_m$ be an orthonormal basis in $L$ and $P$ be a projection on $L.$ Then for any $g_1, \ldots, g_k\in L_2([0; 1])$ the following relation holds
$$
G((I-P)g_1, \ldots, (I-P)g_k)=G(g_1, \ldots, g_k, e_1, \ldots, e_m).
$$
\end{lem}
\begin{proof}
Note that
$$
G((I-P)g_1, \ldots, (I-P)g_k)=G((I-P)g_1, \ldots,(I-P)g_k, e_1, \ldots, e_m).
$$
Put
$$
c_{ij}:=((I-P)g_i, (I-P)g_j)=(g_i,g_j)-\sum^m_{k=1}(e_k, g_i)(e_k,g_j),
$$
then
$$
G((I-P)g_1, \ldots, (I-P)g_k)=G((I-P)g_1, \ldots,(I-P)g_k, e_1, \ldots, e_m)=
$$
\begin{equation}
\label{eq10}
=\begin{vmatrix}
c_{ij}&\vdots&0\\
\hdotsfor{3}\\
0&\vdots&I
\end{vmatrix},\ i,j=\ov{1, k}.
\end{equation}
On other hand
\begin{equation}
\label{eq11}
G(g_1, \ldots, g_k, e_1, \ldots, e_m)=
\begin{vmatrix}
(g_1,g_1)&\ldots&(g_1,g_k)&(g_1,e_1)&\ldots&(g_1,e_m)\\
\vdots&&\vdots&\vdots&&\vdots\\
(g_k,g_1)&\ldots&(g_k,g_k)&(g_k,e_1)&\ldots&(g_k,e_m)\\
(e_1,g_1)&\ldots&(e_1,g_k)&1&\ldots&0\\
\vdots&&\vdots&&\ddots&\\
(e_m,g_1)&\ldots&(e_m,g_k)&0&\ldots&1
\end{vmatrix}.
\end{equation}
Multiplying $(k+1)$-th column by $(g_1, e_1)$, \ldots, $(k+m)$-th column by $(g_1, e_m)$ and subtracting from 1-th column we get that \eqref{eq11} equals
\begin{equation}
\label{eq12}
\begin{vmatrix}
c_{11}&\ldots&(g_1,g_k)&(g_1,e_1)&\ldots&(g_1,e_m)\\
\vdots&&\vdots&\vdots&&\vdots\\
c_{k1}&\ldots&(g_k,g_k)&(g_k,e_1)&\ldots&(g_k,e_m)\\
0&&(e_1,g_k)&1&\ldots&0\\
\vdots&&\vdots&\vdots&&\vdots\\
0&&(e_m,g_k)&0&\ldots&1
\end{vmatrix}
\end{equation}
Multiplying $(k+1)$-th column by $(g_2, e_1), \ldots, $ $(k+m)$-th column by $(g_2, e_m)$ and subtracting from 2-th column we get that \eqref{eq12} equals
\begin{equation}
\label{eq13}
\begin{vmatrix}
c_{11}&c_{12}&\ldots&(g_1,g_k)&(g_1,e_1)&\ldots&(g_1,e_m)\\
\vdots&\vdots&&\vdots&\vdots&&\vdots\\
c_{k1}&c_{k2}&\ldots&(g_k,g_k)&(g_k,e_1)&\ldots&(g_k,e_m)\\
0&0&&(e_1,g_k)&1&\ldots&0\\
\vdots&\vdots&&\vdots&\vdots&&\vdots\\
0&0&&(e_m,g_k)&0&\ldots&1
\end{vmatrix}
\end{equation}
and so on. Finally we get that \eqref{eq11} equals
\begin{equation}
\label{eq14}
\begin{vmatrix}
&&&\vdots&(g_1,e_1)&\ldots&(g_1, e_m)\\
&c_{ij}&&\vdots&\vdots&&\vdots\\
\hdotsfor{7}\\
&0&&\vdots&&I&
\end{vmatrix}.
\end{equation}
Multiplying $(k+1)$-th row by $(g_1, e_1), \ldots, $ $(k+m)$-th row by $(g_1, e_m)$ and subtracting from 1-th row, \ldots,
$(k+1)$-th row by $(g_k, e_k), \ldots, $ $(k+m)$-th row by $(g_k, e_m)$ and subtracting from $k$-th row we get that \eqref{eq14} equals
$$
\begin{vmatrix}
c_{ij}&\vdots&0\\
\hdotsfor{3}\\
0&\vdots&I
\end{vmatrix}, i,j=\ov{1, k}
$$
which proves the lemma.
\end{proof}
\begin{proof}[Proof of Theorem \ref{thm3}] Note that if $P$ is a projection onto $\Ker A,$ then
$$
G(A\1_{[t_1;t_2]}, \ldots, A\1_{[t_{k-1};t_k]} )=
$$
\begin{equation}
\label{eq15}
=G(A(I-P)\1_{[t_1;t_2]}, \ldots, A(I-P)\1_{[t_{k-1};t_k]}).
\end{equation}
Further we need the following statement which was proved in \cite{12}.
\end{proof}
\begin{lem}\cite{12}
\label{lem5}
Suppose that $B$ is a continuously invertible operator in the Hilbert space $H.$ Then for all $k\geq1$ there exists a positive constant $c(k)$ which depends on $k$ and $B$ such that for any $q_1, \ldots, q_k\in H$ the following relation holds
$$
G(Bq_1, \ldots, Bq_k)\geq c(k)G(q_1, \ldots, q_k).
$$
It follows from condition 2) of the theorem and Lemma \ref{lem5} that \eqref{eq15} greater or equal to
\begin{equation}
\label{eq16}
c(k)
G((I-P)\1_{[t_1;t_2]}, \ldots, (I-P)\1_{[t_{k-1};t_k]}).
\end{equation}
\end{lem}
Lemma \ref{lem5} implies that \eqref{eq16} equals
$G(\1_{[t_1;t_2]}, \ldots, \1_{[t_{k-1};t_k]}, e_1, \ldots, e_n)$ for any orthonormal basis $\{e_k,\ k=\ov{1, n}\}$ in $\Ker A.$ Consequently, to prove the theorem it suffices to check that for any $\delta>0$
\begin{equation}
\label{eq17}
\int_{\Delta^\delta_k}
\frac{d\vec{t}}
{G(\1_{[t_1;t_2]}, \ldots, \1_{[t_{k-1};t_k]}, e_1, \ldots, e_n)}<+\infty.
\end{equation}
To check \eqref{eq17} one have to describe the set
$$
\{\vec{t}\in\Delta^\delta_k: G(\1_{[t_1;t_2]}, \ldots, \1_{[t_{k-1};t_k]}, e_1, \ldots, e_n)=0\}.
$$
Note that
$G(\1_{[t_1;t_2]}, \ldots, \1_{[t_{k-1};t_k]}, e_1, \ldots, e_n)=0$ iff there exist $\alpha_1, \ldots, \alpha_{k-1}$ such that $\alpha^2_1+ \ldots+ \alpha^2_{k-1}>0$ and $\beta_1, \ldots, \beta_n$ which satisfy relation
\begin{equation}
\label{eq18}
\sum^{k-1}_{i=1}\alpha_i\1_{[t_i;t_{i+1}]}=\sum^n_{j=1}\beta_je_j.
\end{equation}
Relation \eqref{eq18} implies that if $G(\1_{[t_1;t_2]}, \ldots, \1_{[t_{k-1};t_k]}, e_1, \ldots, e_n)=0,$ then step functions belong to $\Ker A.$ Denote by $L$ the subspace of all step functions in $\Ker A.$ Suppose that $\{f_k,\ k=\ov{1,s}\}$ is an orthonormal basis in $L.$ Let $e_1, \ldots, e_m$ be an orthonormal basis in the orthogonal complement of $L$ in $\Ker A.$ Note that $f_1, \ldots, f_s, e_1, \ldots, e_m$ is an orthonormal basis in $\Ker A$ and for any $\beta_1, \ldots, \beta_m$
$$
\sum^m_{j=1}\beta_je_j\perp L.
$$
Let us check that for any $\delta>0$
\begin{equation}
\label{eq19}
\int_{\Delta^\delta_k}
\frac{d\vec{t}}
{G(\1_{[t_1;t_2]}, \ldots, \1_{[t_{k-1};t_k]},f_1, \ldots, f_s, e_1, \ldots, e_m)}<+\infty.
\end{equation}
To prove \eqref{eq19} we need the following statements.
\begin{lem}
\label{lem6}
Let $M$ be the set of step functions with the amount of jumps less or equal to a fixed number. Then $M$ is a closed set in $L_2([0; 1]).$
\end{lem}
\begin{proof}
Suppose that $M$ is a set of step functions with the number of jumps less or equal to $n.$ Let $\{f_k,\ k\geq1\}\in M$ and $f_k\to f,\ k\to\infty.$ Check that $f\in M.$ Assume that function $f_k$ has jumps in points
$0<t^k_1<\ldots<t^k_{m_k}<1,\ 0\leq m_k\leq n. $ If $m_k=0,$ then $f_k$ does not have jumps. By considering subsequence one can suppose that $m_k=m$ and $(t^k_1, \ldots, t^k_m)\to(t_1, \ldots, t_m),\ k\to\infty,$ where $t_0=0\leq t_1\leq\ldots\leq t_m\leq 1=t_{m+1}.$ Denote by $\pi_{a,b}$ a projection onto $L_2([a; b]).$ If $t_i<t_{i+1}$ for some $i=\ov{0,m}$, then for any $\alpha, \beta$ such that $t_i<\alpha<\beta<t_{i+1}$ the following convergence holds $\pi_{\alpha, \beta}f_k\to\pi_{\alpha,\beta}f,\ k\to\infty.$ Consequently, $f$ is a constant on any $[\alpha;\beta]\subset[t_i; t_{i+1}].$ It implies that $f$ is a constant on $[t_i,t_{i+1}].$ Using the same arguments for any $t_i<t_{i+1}$ we conclude that $f\in M.$
\end{proof}
\begin{lem}
\label{lem7}
There exists a positive constant $c$ such that the following relation holds
$$
G(\1_{[t_1;t_2]}, \ldots, \1_{[t_{k-1};t_k]},f_1, \ldots, f_s, e_1, \ldots, e_m)\geq
$$
$$
\geq
c\ G(\1_{[t_1;t_2]}, \ldots, \1_{[t_{k-1};t_k]},f_1, \ldots, f_s)
$$
\end{lem}
\begin{proof}
Note that
$$
\frac
{G(\1_{[t_1;t_2]}, \ldots, \1_{[t_{k-1};t_k]},f_1, \ldots, f_s, e_1, \ldots, e_m)}
{G(\1_{[t_1;t_2]}, \ldots, \1_{[t_{k-1};t_k]},f_1, \ldots, f_s)}=
$$
$$
=
\frac
{G(\1_{[t_1;t_2]}, \ldots, \1_{[t_{k-1};t_k]},f_1, \ldots, f_s, e_1))}
{G(\1_{[t_1;t_2]}, \ldots, \1_{[t_{k-1};t_k]},f_1, \ldots, f_s)}\cdot
$$
$$
\cdot
\frac
{G(\1_{[t_1;t_2]}, \ldots, \1_{[t_{k-1};t_k]},f_1, \ldots, f_s, e_1,  e_2)}
{G(\1_{[t_1;t_2]}, \ldots, \1_{[t_{k-1};t_k]},f_1, \ldots, f_s, e_1)}\cdot
$$
$$
\cdot\ldots\cdot
\frac
{G(\1_{[t_1;t_2]}, \ldots, \1_{[t_{k-1};t_k]},f_1, \ldots, f_s, e_1, \ldots, e_m)}
{G(\1_{[t_1;t_2]}, \ldots, \1_{[t_{k-1};t_k]},f_1, \ldots, f_s, e_1, \ldots, e_{m-1})}.
$$
Denote  by
$$
\cK^i_t=LS\{\1_{[t_1;t_2]}, \ldots, \1_{[t_{k-1};t_k]},f_1, \ldots, f_s, e_1,\ldots,  e_i\},
$$
$$
\cK^i=\bigcup_{t\in\Delta_k}\cK^i_t.
$$
Let $r_i$ be a distance from $e_i$ to $\cK^i,\ i=\ov{1,m}.$ Then to prove the lemma it suffices to check that for any $i=\ov{1,m} $ \ $r_i>0.$ Suppose that this is not true. Then there exists $j=\ov{1,m}$ such that $r_j=0.$ Let $j=m.$ It implies that there exists the sequence
$$
\Big\{\sum^{k-1}_{i=1}\alpha^n_i\1_{[t^n_i; t^n_{i+1}]}+\sum^s_{j=1}\beta^n_jf_j+\sum^{m-1}_{l=1}\gamma^n_le_l,\  n\geq1\Big\}
$$
such that
$$
\Big\|
e_m-\sum^{k-1}_{i=1}\alpha^n_i\1_{[t^n_i; t^n_{i+1}]}-\sum^s_{j=1}\beta^n_jf_j-\sum^{m-1}_{l=1}\gamma^n_le_l
\Big\|\to0,\ n\to\infty.
$$
Therefore, the question is when
$$
\sum^{k-1}_{i=1}\alpha^n_i\1_{[t^n_i; t^n_{i+1}]}+\sum^s_{j=1}\beta^n_jf_j+\sum^{m-1}_{l=1}\gamma^n_le_l+e_m
$$
tends to zero as $n\to\infty$ ? Consider possible cases.

1) Suppose that
$$
\varlimsup_{n\to\infty}\left\|\sum^{m-1}_{l=1}\gamma^n_le_l+e_m\right\|<+\infty.
$$
Considering a subsequence assume that for $l=\ov{1, m-1}$ $\gamma^n_l\to\gamma_l,\ n\to\infty.$
It implies that
$$
\sum^{k-1}_{i=1}\alpha^n_i\1_{[t^n_i; t^n_{i+1}]}+\sum^s_{j=1}\beta^n_jf_j\to-\sum^{m-1}_{l=1}\gamma_le_l-e_m,\  n\to\infty.
$$
Note that $\sum^{m-1}_{l=1}\gamma_le_l+e_m$ is not a step function. On other hand
$$
\Big\{
\sum^{k-1}_{i=1}\alpha^n_i\1_{[t^n_i; t^n_{i+1}]}+\sum^s_{j=1}\beta^n_jf_j,\ n\geq1
\Big\}
$$
is a sequence of step functions with the number of jumps less or equal to a fixed number. It follows from Lemma \ref{lem6} that situation 1) is impossible.

2) Considering a subsequence one can suppose that
$$
a_n=\|\sum^{m-1}_{l=1}\gamma^n_le_l+e_m\|\to+\infty,\ n\to+\infty
$$
and
$$
\frac{1}{a_n}\Big(\sum^{m-1}_{l=1}\gamma^n_le_l+e_m\Big)\to\sum^{m-1}_{l=1}p_le_l, \ n\to\infty,
$$
where $\|\sum^{m-1}_{l=1}p_le_l\|=1.$ Using the same arguments as in case 1) one can get a contradiction.
\end{proof}
\begin{lem}
\label{lem8}
Let $0\leq s_1<\ldots<s_N\leq1$ be the points of jumps of functions $f_1, \ldots, f_s.$ Then there exists a positive constant $c_{\vec{s}}$ which depends on $\vec{s}=(s_1, \ldots, s_N)$ such that the following relation holds
$$
G(\1_{[t_1;t_2]}, \ldots, \1_{[t_{k-1};t_k]},f_1, \ldots, f_s)\geq
$$
\begin{equation}
\label{eq20}
\geq c_{\vec{s}}\ G(\1_{[t_1;t_2]}, \ldots, \1_{[t_{k-1};t_k]}, \1_{[s_1;s_2]}, \ldots, \1_{[s_{N-1};s_N]}).
\end{equation}
\end{lem}
\begin{proof}
Note that for any $i=\ov{1,s}$
$$
f_i\in LS\{\1_{[s_j; s_{j+1}]},\ j=\ov{1, N-1}\}.
$$
Let us prove the statement of the lemma by induction.
Let $d_i$ be a distance from $\1_{[t_i; t_{i+1}]}$ to $LS\{\1_{[t_{i+1};t_{i+2}]}, \ldots,\1_{[t_{k-1};t_{k}]},f_1,\ldots, f_s \}$ and $\rho_i$ be a distance from $\1_{[t_i; t_{i+1}]}$ to\break $LS\{\1_{[t_{i+1};t_{i+2}]}, \ldots,\1_{[t_{k-1};t_{k}]},
\1_{[s_{1};s_{2}]}, \ldots,\1_{[s_{N-1};s_{N}]} \}.$ Then in case $k=2$
$$
G(\1_{[t_1;t_2]},f_1, \ldots, f_s)=d^2_1\geq
\rho^2_1\
G\Big(
\frac{\1_{[s_1;s_2]}}
{\sqrt{s_2-s_1}}, \ldots, \frac{\1_{[s_{N-1};s_N]}}
{\sqrt{s_N-s_{N-1}}}\Big)=
$$
$$
=
c_{\vec{s}}\
G(\1_{[t_1;t_2]}, \1_{[s_{1};s_{2}]}, \ldots,\1_{[s_{N-1};s_{N}]}).
$$
Assume that \eqref{eq20} holds for $k.$ Then
$$
G(\1_{[t_1;t_2]}, \ldots, \1_{[t_{k-1};t_k]},f_1, \ldots, f_s)=d^2_1\
G(\1_{[t_2;t_3]}, \ldots, \1_{[t_{k-1};t_k]},f_1, \ldots, f_s)\geq
$$
$$
\geq\rho^2_1\  c_{\vec{s}}\ G(\1_{[t_2;t_3]}, \ldots, \1_{[t_{k-1};t_k]}, \1_{[s_{1};s_{2}]}, \ldots,\1_{[s_{N-1};s_{N}]})=
$$
$$
=
c_{\vec{s}}\
G(\1_{[t_1;t_2]}, \ldots, \1_{[t_{k-1};t_k]}, \1_{[s_{1};s_{2}]}, \ldots,\1_{[s_{N-1};s_{N}]}).
$$
\end{proof}
\begin{lem}
\label{lem9}
For an arbitrary $0<s_1<\ldots<s_N<1,\ \delta>0$ the following integral
$$
\int_{\Delta^\delta_k}
\frac{1}
{G(\1_{[t_1;t_2]}, \ldots, \1_{[t_{k-1};t_k]}, \1_{[s_{1};s_{2}]}, \ldots,\1_{[s_{N-1};s_{N}]}}dt_1\ldots dt_k
$$
converges.
\end{lem}
\begin{proof}
Denote by
$$
G(\1_{[t_1;t_2]}, \ldots, \1_{[t_{k-1};t_k]}, \1_{[s_{1};s_{2}]}, \ldots,\1_{[s_{N-1};s_{N}]})=G(\vec{t}, \vec{s}), \ \vec{t}\in\Delta^\delta_k.
$$
Note that
$$
\frac{1}{(2\pi)^{k+N-2}}\frac{1}{G(\vec{t}, \vec{s})}=
$$
$$
=E\prod^{k-1}_{j=1}\delta_0(w(t_{j+1})-w(t_j))\prod^{N-1}_{i=1}\delta_0(w(s_{i+1})-w(s_i))=
$$
$$
=
E\int_{\mbR^2}\prod^{k}_{j=1}\delta_0(w(t_{j})-u)du\int_{\mbR^2}\prod^{N}_{i=1}\delta_0(w(s_{i})-v)dv,
$$
where $w$ is a planar Wiener process. Denote by $r_1<r_2<\ldots<r_{k+N}$ the points $t_1, \ldots, t_k, s_1, \ldots, s_N$ which are ordered by increasing. Put
$$
\theta(r_i)=\begin{cases}
u, &r_i\in\{t_1, \ldots, t_k\}\\
v, &r_i\in\{s_1, \ldots, s_N\}
\end{cases}.
$$
Then
$$
E\int_{\mbR^2}\prod^{k}_{j=1}\delta_0(w(t_{j})-u)du\int_{\mbR^2}\prod^{N}_{i=1}\delta_0(w(s_{i})-v)dv=
$$
$$
=
\int_{\mbR^{2\times2}}p_{r_1}(\theta(r_1))
\prod^{k+N-1}_{l=1}p_{r_{l+1}-r_l}(\theta(r_{l+1})-\theta(r_l))dudv.
$$
Here $p_q(y)=\frac{1}{2\pi q}e^{-\frac{\|y\|^2}{2q}},\ q\in[0; 1],\ y\in\mbR^2.$ Let us check that the following integral
$$
\int_{\Delta^\delta_k}p_{r_1}(\theta(r_1))
\prod^{k+N-1}_{l=1}p_{r_{l+1}-r_l}(\theta(r_{l+1})-\theta(r_l))d\vec{t}
$$
converges. It suffices to fix some order of $\{r\}$ and check the integrability on that fixed subset of $\Delta^\delta_k$. For example,
$\{r\}=\{0<t_1<s_1<t_2<t_3<t_4<s_2<\ldots\}.$ One can check that the following estimates hold.

1) $$
\int^1_{s_N}p_{t_k-s_N}(u-v)dt_k=\frac{1}{2\pi}\int^1_{s_N}
\frac
{e^{-\frac{\|u-v\|^2}{2(t_k-s_N)}}}
{t_k-s_N}
dt_k\leq
$$
$$
\leq c\int^{+\infty}_1\frac{1}{r}e^{-\frac{r\|u-v\|^2}{2}}dr=
$$
$$
=c\int^{+\infty}_{\frac{\|u-v\|^2}{2}}
\frac{1}{r}e^{-r}dr\leq
$$
$$
\leq
\wt{c}\ (\ln\|u-v\|\1_{\big\{\frac{\|u-v\|^2}{2}<1\big\}})+e^{-\frac{\|u-v\|^2}{2}}\1_{\big\{\frac{\|u-v\|^2}{2}<1\big\}}=f_0(\|u-v\|),
$$
where $c,\  \wt{c}$ are positive constants and
$$
\int_{\mbR^2}f_0(u)^pdu<+\infty
$$
for all $p>0.$

2)
$$
\int^1_{t_{k-1}+\delta}p_{t_k-t_{k-1}}(0)dt_k=\int^1_{t_{k-1}+\delta}
\frac{1}{2\pi(t_k-t_{k-1})}dt_k\leq c\ \ln\delta
$$

3)
$$
\int^{s_N}_{s_{N-1}}p_{t_k-s_{N-1}}(u-v)p_{s_N-t_{k}}(u-v)dt_k\leq
$$
$$
\leq
c\ \frac{2}{s_N-s_{N-1}}\int^{s_N}_{s_{N-1}}(u-v)dt_k\leq\wt{c}\ f_0(u-v);
$$

4)
$$
\int^{s_N}_{t_{k-1}+\delta}p_{t_k-t_{k-1}}(0)p_{s_N-t_k}(u-v)dt_k\leq
$$
$$
\leq c\ \frac{1}{\delta}\int^{s_N}_{t_{k-1}}p_{s_N-t_k}(u-v)dt_k\leq\wt{c}\ f_0(u-v).
$$
Therefore, if $s_1<t_1,$ then
$$
\int_{\{r\}}p_{r_1}(\theta(r_1))\prod^{k+N-1}_{l=1}p_{r_{l+1}-r_l}(\theta(r_{l+1})-\theta(r_l))d\vec{s}\leq
$$
$$
\leq c\ p_{s_1}(v)f_0(u-v)^m,
$$
where constant $c$ can depend on $\vec{s}$ and $m$ equals to a number of pairs $s$ and $t$ in sequence $r_1,\ldots,r_n.$

If $t_1<s_1,$ then
$$
\int_{\{r\}}p_{r_1}(\theta(r_1))\prod^{k+N-1}_{l=1}p_{r_{l+1}-r_l}(\theta(r_{l+1})-\theta(r_l))d\vec{s}\leq
$$
$$
\leq c\int^1_0p_{t_1}(u)f_0(u-v)^mdt_1\leq\wt{c}f_0(u)f_0(u-v)^m.
$$
Consequently,
$$
\int_{\mbR^{2\times2}}\int_{\{r\}}p_{r_1}(\theta(r_1))\prod^{k+N-1}_{l=1}p_{r_{l+1}-r_l}(\theta(r_{l+1})-\theta(r_l))
d\vec{s}dudv\leq
$$
$$
\leq \begin{cases}
c\int_{\mbR^{2\times2}}p_{s_1}(v)f_0(u-v)^mdudv\\
\wt{c}\int_{\mbR^{2\times2}}f_0(u)f_0(u-v)^mdudv,
\end{cases}
$$
where integrals in last estimate converge.
\end{proof}
The following statements related to $k$-multiple SILT for planar Gaussian integrator are consequences of Theorem \ref{thm3}.
\begin{thm}
\label{thm4}
Let $x$ be a planar Gaussian integrator generated by continuous linear operator in $L_2([0; 1])$ which satisfies conditions 1), 2). Then there exists the Fourier--Wiener transform of $k$-multiple SILT of $x$ on domain off the diagonals.
\end{thm}

Denote by
$$
T^x_{\ve,k,\delta}=\int_{\Delta^\delta_k}\prod^{k-1}_{i=1}f_\ve(x(t_{i+1})-x(t_i))d\vec{t}.
$$
\begin{thm}
\label{thm5}
Let $x$ be a planar Gaussian integrator generated by continuous linear operator $A$ in $L_2([0;1])$ which satisfies conditions 1), 2). Then there exists
$$
T^x_{k,\delta}=L_2\mbox{-}\lim_{\ve\to0}T^x_{\ve, k, \delta}.
$$
\end{thm}
\begin{proof}
To prove the theorem it suffices to check that there exists finite limit of $ET^x_{\ve_1,k,\delta}T^x_{\ve_2,k,\delta}$ as $\ve_1, \ve_2\to0.$

Note that
$$
ET^x_{\ve_1,k,\delta}T^x_{\ve_2,k,\delta}=E\int_{\Delta^\delta_k}\prod^{k-1}_{i=1}f_{\ve_1}(x(t_{i+1})-x(t_i))d\vec{t}
$$
$$
\int_{\Delta^\delta_k}\prod^{k-1}_{j=1}f_{\ve_2}(x(s_{j+1})-x(s_j))d\vec{s}=
$$
\begin{equation}
\label{eq21}
=
\int_{\Delta^\delta_k\times\Delta^\delta_k}
\frac{1}
{(2\pi)^{2k-2}\det(C_{t_1\ldots t_ks_1\ldots s_k}+I(\ve_1, \ve_2))}d\vec{t}d\vec{s},
\end{equation}
where $I(\ve_1, \ve_2)$ is the following matrix
$$
\begin{pmatrix}
\ve_1&&0&\vdots&&&\\
&\ddots&&\vdots&&0&\\
0&&\ve_1&\vdots&&&\\
\hdotsfor{7}\\
&&&\vdots&\ve_2&&0\\
&0&&\vdots&&\ddots&\\
&&&\vdots&0&&\ve_2
\end{pmatrix}.
$$
Here $C_{t_1\ldots t_ks_1\ldots s_k}$ is the Gramian matrix constructed from
$$
A\1_{[t_1;t_2]}, \ldots, A\1_{[t_{k-1}; t_k]},A\1_{[s_1;s_2]}, \ldots, A\1_{[s_{k-1}; s_k]}.
$$
One can check that \eqref{eq21} less or equal to
\begin{equation}
\label{eq22}
\int_{\Delta^\delta_k\times\Delta^\delta_k}
\frac{d\vec{t}d\vec{s}}
{G(A\1_{[t_1;t_2]}, \ldots, A\1_{[t_{k-1}; t_k]},A\1_{[s_1;s_2]}, \ldots, A\1_{[s_{k-1}; s_k]})}.
\end{equation}
The same arguments as in Theorem \ref{thm3} lead to the finiteness of the last integral. Now the dominated convergence theorem finishes the proof.
\end{proof}

\vskip30pt
\noindent
Institute of Mathematics,
National Academy of Sciences of Ukraine,\newline
Tereshchenkivska Str. 3, Kiev 01601, Ukraine
\vskip15pt
\noindent
{\it E-mail address}: adoro@imath.kiev.ua

\noindent
{\it E-mail address}: olaizyumtseva@yahoo.com

\end{document}